\documentclass[11pt,leqno]{amsart}
\usepackage{amsthm,amsfonts,amssymb,amsmath,oldgerm}
\numberwithin{equation}{section}
\usepackage{fullpage}
\usepackage{color}
\usepackage{calrsfs}
\DeclareMathAlphabet{\pazocal}{OMS}{zplm}{m}{n}

\newcommand{\Hs}{{\bH_*}}

\newcommand{\CalD}{\mathcal{D}}

\def\eps{\varepsilon }

\newcommand\R{\mathbb R}

\def\eps{\varepsilon}

\newcommand\br{\begin{remark}}
\newcommand\er{\end{remark}}
\newcommand\bp{\begin{pmatrix}}
\newcommand\ep{\end{pmatrix}}
\newcommand{\be}{\begin{equation}}
\newcommand{\ee}{\end{equation}}
\newcommand\ba{\begin{equation}\begin{aligned}}
\newcommand\ea{\end{aligned}\end{equation}}

\newcommand{\bap}{\begin{app}}
\newcommand{\eap}{\end{app}}
\newcommand{\begs}{\begin{exams}}
\newcommand{\eegs}{\end{exams}}
\newcommand{\beg}{\begin{example}}
\newcommand{\eeg}{\end{exaplem}}
\newcommand{\bpr}{\begin{proposition}}
\newcommand{\epr}{\end{proposition}}
\newcommand{\bt}{\begin{theorem}}
\newcommand{\et}{\end{theorem}}
\newcommand{\bc}{\begin{corollary}}
\newcommand{\ec}{\end{corollary}}
\newcommand{\bl}{\begin{lemma}}
\newcommand{\el}{\end{lemma}}
\newcommand{\bd}{\begin{definition}}
\newcommand{\ed}{\end{definition}}
\newcommand{\brs}{\begin{remarks}}
\newcommand{\ers}{\end{remarks}}

\newcommand{\RR}{{\mathbb R}}

\newcommand{\CC}{{\mathbb C}}

\newcommand{\pa}{{\partial}}
\newcommand{\rmi}{{\mathrm{i}}}
\newcommand{\rmd}{{\mathrm{d}}}
\newcommand{\rms}{{\mathrm{s}}}
\newcommand{\rmu}{{\mathrm{u}}}

\newcommand{\Id}{{\rm Id }}

\newcommand{\Range}{{\rm Range }}

\newtheorem{theorem}{Theorem}[section]
\newtheorem{proposition}[theorem]{Proposition}
\newtheorem{corollary}[theorem]{Corollary}
\newtheorem{lemma}[theorem]{Lemma}

\theoremstyle{remark}
\newtheorem{remark}[theorem]{Remark}
\theoremstyle{definition}
\newtheorem{definition}[theorem]{Definition}

\newtheorem{example}[theorem]{Example}

\newcommand\cB{{\mathcal B}}
\newcommand\cD{{\mathcal D}}

\newcommand\cJ{{\mathcal J}}

\newcommand\cK{{\mathcal K}}

\newcommand\cE{{\mathcal E}}
\newcommand\cF{{\mathcal F}}

\newcommand\cM{{\mathcal M}}

\newcommand\cS{{\mathcal S}}
\newcommand\cZ{{\mathcal Z}}

\newcommand\bB{{\mathbb B}}

\newcommand\bH{{\mathbb H}}

\newcommand\bX{{\mathbb X}}

\newcommand{\ovB}{{\overline{B}}}

\newcommand{\bu}{\mathbf{u}}

\newcommand{\dom}{\text{\rm{dom}}}

\newcommand{\supp}{\text{\rm{supp}}}

\newcommand{\beq}{\begin{equation}}
\newcommand{\eeq}{\end{equation}}

\title{
Reverse norms and $L^\infty$ exponential decay for a class of degenerate evolution systems arising in kinetic theory
}

\author{ Alin Pogan}
\address{Miami University, Oxford, OH 45056}
\email{pogana@miamioh.edu}
\thanks{ A. P. research was partially supported under the
Summer Research Grant program, Miami University}
\author{Kevin Zumbrun}
\address{Indiana University, Bloomington, IN 47405}
\email{kzumbrun@indiana.edu}
\thanks{Research of K.Z. was partially supported
under NSF grant no. DMS-0300487}

\begin{document}

\begin{abstract}
We consider the question of exponential decay to equilibrium of solutions of an abstract class of degenerate evolution
equations on a Hilbert space modeling the steady Boltzmann and other kinetic equations.
Specifically, we provide conditions suitable for
construction of a stable manifold in a particular ``reverse $L^\infty$'' norm and examine when these do and do not hold.
\end{abstract}

\maketitle

\vspace{0.3cm}
\begin{minipage}[h]{0.48\textwidth}
\begin{center}
Miami University\\
Department of Mathematics\\
301 S. Patterson Ave.\\
Oxford, OH 45056, USA
\end{center}
\end{minipage}
\begin{minipage}[h]{0.48\textwidth}
\begin{center}
Indiana University \\
Department of Mathematics\\
831 E. Third St.\\
Bloomington, IN 47405, USA
\end{center}
\end{minipage}

\vspace{0.3cm}

\tableofcontents

\section{Introduction }\label{s1}
In this note, we study decay to zero for a class of degenerate evolution equations
\begin{equation}\label{Gamma}
\Gamma u'+ u = D(u,u)
\end{equation}
on a real Hilbert space $\bH$, where $D(\cdot,\cdot)$ is a bounded bilinear map on $\bH$
and $\Gamma$ is a (fixed) bounded linear operator on $\bH$ that
is {\it self-adjoint, one-to-one, but not invertible.}
As described in \cite[Introduction, Eqs. (1.5), (1.8)]{PZ1}, this is equivalent to the study
of decay to equilibria of shock and boundary layer solutions of a class of kinetic equations including
the steady Boltzmann equation.

This equation is most easily understood via spectral decomposition of $\Gamma$, converting \eqref{Gamma}
to a family of scalar equations
\be\label{scal}
(\gamma_\lambda \partial_x + 1)u_\lambda=D_\lambda(u,u),
\ee
indexed by $\lambda\in\Lambda$, where $u_\lambda$ is the coordinate of $u$ associated with spectrum $\gamma_\lambda$, real,
in the eigendecomposition of $\Gamma $, with
$\|u\|_{\bH}^2=\int_\Lambda |u_\lambda|^2d\mu_\lambda$. Here $d\mu$ denotes spectral measure associated with
$\lambda$, and the set $\{\gamma_\lambda:\lambda\in\Lambda\}$ is bounded with an accumulation point at $0$.

The linearized equations about zero, $(\gamma_\lambda \partial_x+1)u_\lambda=0$, have a stable subspace
consisting of
\be\label{stable}
u_\lambda(x)= u_\lambda(0)e^{-x/\gamma_\lambda},
\ee
for any $u(0)\in \bH$ with $u_\lambda(0)=0$ whenever $\gamma_\lambda<0$, satisfying the uniform exponential bound
\begin{equation*}
\|u(x)\|_{\bH}\leq Ce^{-x/\sup_{\lambda\in\Lambda}|\gamma_\lambda|}\|u(0)\|_{\bH}.
\end{equation*}
On the other hand, derivatives $\gamma_\lambda^{-1}e^{-x/\gamma_\lambda}$ of functions in the stable subspace,
being multiplied by the unbounded factor $\gamma_\lambda^{-1}$, may not even lie in $\bH$.
We denote as the {\it $H^1$ stable subspace} the subset of solutions in the stable subspace that
are contained in $H^1(\R_+, \bH)$, namely, those with
$$
u(0)\in \dom(\Gamma^{-1/2}).
$$
See \cite{PZ1} for further details.

In \cite{PZ1}, it was shown that there exists an $H^1$ stable {\it manifold}, consisting of all solutions of the full
equation that are sufficiently small in $H^1(\R_+,\bH)$, lying tangent to the $H^1$ stable subspace, on which solutions
decay uniformly exponentially in $H^1(\R_+,\bH)$, hence decay pointwise at rate $e^{-\beta x}$, for some $\beta>0$.
However, it was also shown that the linearized solution operator $(\Gamma \partial_x + \Id)^{-1}$, though a bounded
operator on all $L^p(\R_+,\bH)$, $1<p<\infty$, is {\it unbounded} on $L^1(\R_+,\bH)$ and $L^\infty(\R_+,\bH)$, as
a consequence of which the usual stable manifold construction fails in $L^\infty(\R_+,\bH)$.
It was cited as an interesting open problem in \cite{PZ1} whether there exists a ``full'' stable manifold, tangent to the entire
linear stable subspace.
Here, we consider a specific approach to this problem based on the introduction of a nonstandard ``reverse norm.''

\subsection{The reverse norm}\label{s:reverse}
Let us first review the standard fixed point construction of the stable manifold
(as carried out for finite dimensions in, e.g., \cite{GH,HK}) within the
context considered  in \cite{PZ1}.
Define $\chi_+(x)$ to be the cutoff function returning $1$ for $x\geq 0$ and $0$ otherwise, and
let $\Pi_\rms$ be projection onto the stable subspace of $\Gamma$ and $T_\rms(\cdot)$ the semigroup induced by \eqref{Gamma}
restricted to the stable subspace of $-\Gamma^{-1}$.
Then, solutions of \eqref{Gamma} on $\R_+$ may be expressed as
\be\label{intvar}
u = T_\rms(\cdot)\Pi_\rms u(0) + \chi_+(\Gamma \partial_x+\Id)^{-1}\big(\chi_+D_*(u,u)\big),
\ee
where $D_*:L^2(\RR_+,\bH)\times L^2(\RR,\bH)\to L^1(\RR,\bH)$ is the bilinear map defined by $D_*(u,v)=D(u(\cdot),v(\cdot))$.
Equation \eqref{intvar} can be seen as a variant of the usual variation of constants formula,
so long as the inverse $(\Gamma \partial_x+\Id)^{-1}$ is well-defined on functions of the form $\chi_+ D_*(u,u)$.
Indeed, with some elaborations,
\eqref{intvar} is used in \cite{LP2,LP3,PZ1} effectively as the definition of a mild solution of \eqref{Gamma} on $\RR_+$.

What we seek, then, is a Banach Space $\cZ$ of functions on $\R_+$ that is continuously embedded in $L^\infty(\RR,\bH)$, closed under the action of $(\Gamma\partial_x+\Id)^{-1}$ and of $D_*$, in the sense that
\begin{equation}\label{closed-D*}
D_*(v,v)\in\cZ\;\mbox{for any}\;v\in\cZ\;\mbox{and}\;\|D_*(v,v)\|_{\cZ}\lesssim\|v\|_{\cZ}^2\;\mbox{for any}\;v\in\cZ.
\end{equation}
Moreover, the space $\cZ$ should be large enough to contain the subspace of trajectories $\{T_\rms(\cdot)h:h\in\Range\,\Pi_\rms\}$.
When these properties hold, one readily sees that the fixed-point equation \eqref{intvar} is a contraction, yielding existence
and uniqueness of the stable manifold;
for details of the construction, definition of mild solution, etc., see the similar analysis of \cite{PZ1}.

To this end, we introduce the {\it reverse norm}
\be\label{reverse}
\|u\|_\Hs:= \Big(\int_{\Lambda} (\sup_{x\in\R_+} |u_\lambda(x)|)^2 d\mu_\lambda \Big)^{1/2},
\ee
where $d\mu$ denotes spectral measure associated with $\lambda$, and the space $\Hs$ of functions on $\R_+$
with finite $\|\cdot\|_\Hs$ norm.
We see readily that $(\Gamma\partial_x+\Id)$
is boundedly invertible on $\Hs$, with resolvent kernel given in $u_\lambda$ coordinates by
the scalar resolvent kernel
\be \label{R}
\hbox{\rm $R_\lambda (x,y)= \gamma_\lambda^{-1}e^{(x-y)/\gamma_\lambda^{-1}}$ whenever $(x-y)\gamma_\lambda<0$,}
\ee
and $0$ otherwise.
For, \eqref{R} is {\it integrable with respect to $x$}, hence bounded coordinate-by-coordinate with respect
to $L^\infty(\R_+)$, as therefore is the square integral $\|\cdot\|_\Hs$ of all coordinates.

The question thus reduces to: {\it ``under what conditions on $D$ is the extension $D_*$
closed with respect to $\Hs$?''} When such conditions are met, we have existence of a unique stable manifold in $\Hs$, tangent
to the full stable subspace of the bi-semigroup generated by $-\Gamma^{-1}$, answering the open question of \cite{PZ1};
see Theorem \ref{t} and Corollary \ref{c}.

\subsection{Results and counterexamples}\label{res}

We first look for a sharp abstract condition that characterizes condition \eqref{closed-D*}. For any $\alpha\in L^2(d\mu_\lambda)$ we introduce the set
\begin{equation}\label{def-calE}
\cE_\alpha=\big\{v\in\bH:|v_\lambda|\leq |\alpha_\lambda|\;\mbox{for any}\;\lambda\in\Lambda\big\}.
\end{equation}
We recall that here $(v_\lambda)_{\lambda\in\Lambda}$ denote the spectral coordinates of $v\in\bH$. Next, for any $\alpha\in L^2(d\mu_\lambda)$ we define  $\cS(\alpha)$ by
\begin{equation}\label{def-calS}
\cS_\lambda(\alpha)=\sup_{v\in\cE_\alpha}|D_\lambda(v,v)|.
\end{equation}
The following three results are established in Section~\ref{results}.
\bpr\label{abcond}
The reversed-norm space $\Hs$ is closed under the action of $D_*$ if and only if
\begin{equation}\label{bound-calS}
\|\cS(\alpha)\|_{L^2(d\mu_\lambda)}\lesssim\|\alpha\|_{L^2(d\mu_\lambda)}^2\;\mbox{for any}\;\alpha\in L^2(d\mu_\lambda).
\end{equation}
\epr
Let $D$ be expressed in terms of a kernel $\CalD:\Lambda^3\to\RR$, via
\be\label{kerD}
D_\lambda(v,w)=\int_{\Lambda^2} \CalD(\lambda, \nu,\sigma) v_\nu w_\sigma \, d\mu_\nu  \,d\mu_\sigma .
\ee
Then, two sufficient conditions are as follows.

\bpr[Hilbert-Schmidt condition]\label{fredprop}
The reversed-norm space $\Hs$ is closed under the action of $D_*$ provided
\be\label{fredeq}
\int_{\Lambda^3} |\CalD(\lambda, \nu,\sigma)|^2 \, d\mu_\lambda \, d\mu_\nu  \, d\mu_\sigma <+\infty.
\ee
\epr
\bpr[Absolute boundedness condition]\label{abprop}
The reversed-norm space $\Hs$ is closed under the action of $D_*$ provided the bilinear map $|D|$ with kernel $|\CalD(\lambda,\gamma,\sigma)|$
is bounded from $\bH\times\bH\to\bH$.
\epr
\br\label{fredrmk}
In \cite{PZ1}, the form of equation \eqref{Gamma} arose through linearization about $\bar u\in \bH$ of
\be\label{origin}
\Gamma u'=\tilde D(u,u),
\ee
$\tilde D$ a bounded bilinear map;
that is, the term $u$ on the lefthand side of \eqref{Gamma} corresponds to the relation $-2Id=D(\bar u, \cdot )$.
But, this contradicts \eqref{fredeq}, since \eqref{fredeq} together with the Cauchy-Schwarz inequality gives
$$
\|\CalD\|_{L^2(d\mu^3)} \leq \|\bar u\|_{\bH}^2 \|\tilde D\|_{\bH\to \bH}^2<\infty.
$$
Thus, the Hilbert-Schmidt criterion, though appealing, is not relevant to the problem originally considered in \cite{PZ1},
in particular not to the case of the steady Boltzmann equation.
\er
Using these results we can prove the following existence and uniqueness result.
\begin{theorem}\label{t}
	Assume that $\Hs$ is closed under the action of $D_*$ (for example, that
	$|D|$ is bounded from $\bH\times\bH$ to $\bH$, or \eqref{fredeq} is satisfied).
Then, for any integer $r\geq 1$ there exists a $C^r$ local stable manifold $\cM_\rms$ near $0$, expressible as a graph of $C^r$ function $\cJ_\rms:\bH_\rms\to\bH_\rmu$, that is locally invariant under the flow of equation $\Gamma u'+u=D(u,u)$ and uniquely determined by the property that $u\in \Hs$.
\end{theorem}
Using the existence result above we obtain the following exponential decay result for solutions of equation \eqref{Gamma}:
\begin{corollary}\label{c}
Assume that $\Hs$ is closed under the action of $D_*$, and let $u^*\in\bH_*$ be a solution of equation $\Gamma u'+u=D(u,u)$.
Then, there exist $\beta\in (0,\|\Gamma\|^{-1})$ such that $u^*\in \bH_{*,\beta}$. In particular, we have that there exists $\beta>0$ such that $\|u^*(x)\|\lesssim e^{-\beta|x|}$ for any $x\in\RR_\pm$.
\end{corollary}

It is straightforward to construct kernels $\CalD$ originating from linearization of
\eqref{origin} and satisfying the condition of Proposition~\ref{abprop} but not \eqref{fredeq}.
Hence, Proposition~\ref{abprop} gives existence of a full stable manifold in some cases relevant to the scenario
originally considered in \cite{PZ1}.
However, this condition too is not sharp. In Section~\ref{counterexamples} we give two counterexamples showing that Propositions~\ref{fredprop} and ~\ref{abprop} provide only sufficient conditions guaranteeing that $\Hs$ is closed under the action of $D_*$.  Finally, in Section \ref{discussion}, we discuss possible generalizations, and open problems, in particular
as regards the important example of the steady Boltzmann equation, our main interest in \cite{PZ1}.

\section{Results}\label{results}

In this section we prove our results stated in the introduction. First, we give necessary and sufficient conditions that guarantee that the extension $D_*$ is bounded.
Then, we sketch the proof of existence and uniqueness of a stable manifold of equation \eqref{Gamma} assuming
boundedness of $D_*$.

\subsection{Invariance of $\Hs$ under the action of $D_*$}\label{sufficient-cond}
\begin{proof}\textit{of Proposition~\ref{abcond}}. First, we assume that condition \eqref{closed-D*} holds for $\cZ=\Hs$. Fix $\alpha\in L^2(d\mu_\lambda)$. From the definition of $\cS(\alpha)$ in \eqref{def-calS} we have that for any $\lambda\in\Lambda$ there exists $v^\lambda\in\cE_\alpha\subset\bH$ such that
\begin{equation}\label{2.1-1}
\frac{1}{2}\cS_\lambda(\alpha)\leq |D_\lambda(v^\lambda,v^\lambda)|\;\mbox{for any}\;\lambda\in\Lambda.
\end{equation}
Since the linear operator $\Gamma$ is self-adjoint, its spectral decomposition $\Lambda$ is contained in $\RR$. It follows that for any $\lambda\in\Lambda$ there exist $x_\lambda\in\RR_+$ such that $x_\lambda\ne x_\nu$ for any $\lambda\ne\nu$. Next, for any $\lambda\in\Lambda$ we construct a function $w_\lambda:\RR_+\to\CC$ such that $w_\lambda(x_\nu)=v_\lambda^\nu$ for any $\nu\in\Lambda$ and $0$ otherwise. From \eqref{def-calE} we obtain that
\begin{equation}\label{2.1-2}
\sup_{x\geq 0}|w_\lambda(x)|=\sup_{\nu\in\Lambda}|v_\lambda^\nu|\leq |\alpha_\lambda|\;\mbox{for any}\;\lambda\in\Lambda.
\end{equation}
Since $(v_\lambda^\nu)_{\lambda\in\Lambda}\in L^2(d\mu_\lambda)$  we have that the function $\lambda\to v_\lambda^\nu$ is $\mu$-measurable for any $\nu\in\Lambda$, which implies that
the function $\lambda\to \sup_{x\geq 0}|w_\lambda(x)|$ is $\mu$-measurable. Since $\alpha\in L^2(d\mu_\lambda)$ from \eqref{2.1-2} we conclude that $w=(w_\lambda)_{\lambda\in\Lambda}\in\Hs$ and $\|w\|_{\Hs}\leq \|\alpha\|_{L^2(d\mu_\lambda)}$. Moreover, we have that $w(x_\lambda)=v^\lambda$ for any $\lambda\in\Lambda$. From \eqref{2.1-1} we obtain that
\begin{equation}\label{2.1-3}
\sup_{x\geq 0}|D_\lambda(w(x),w(x)|\geq |D_\lambda(v^\lambda,v^\lambda)|\geq \frac{1}{2}\cS_\lambda(\alpha)\;\mbox{for any}\;\lambda\in\Lambda.
\end{equation}
Since $\Hs$ is closed under the action of the extension $D_*$, from \eqref{closed-D*} and \eqref{2.1-3} it follows that
\begin{equation}\label{2.1-4}
\int_{\Lambda} |\cS_\lambda(\alpha)|^2d\mu_\lambda\leq 2\int_{\Lambda}\big(\sup_{x\geq 0}|D_\lambda(w(x),w(x)|\big)^2d\mu_\lambda=2\|D_*(w,w)\|_{\Hs}^2\lesssim\|w\|_{\Hs}^2\leq \|\alpha\|_{L^2(d\mu_\lambda)}^2,
\end{equation}
proving that $\cS(\alpha)\in L^2(d\mu_\lambda)$ and \eqref{bound-calS} holds true. Conversely, assume \eqref{bound-calS} and let $w=(w_\lambda)\in\Hs$. From the definition of $\Hs$ we immediately infer that $\alpha=(\alpha_\lambda)_{\lambda\in\Lambda}$ defined by $\alpha_\lambda=\sup_{x\geq 0}|w_\lambda(x)|$ belongs to $L^2(d\mu_\lambda)$. Moreover, from \eqref{def-calE} and \eqref{def-calS}, respectively, we have that $w(x)\in\cE_\alpha$ and $|D_\lambda(w(x),w(x))|\leq \cS_\lambda(\alpha)$ for any $x\geq 0$ and $\lambda\in\Lambda$. We conclude that
\begin{equation}\label{2.1-5}
\int_{\Lambda}\big(\sup_{x\geq 0}|D_\lambda(w(x),w(x)|\big)^2d\mu_\lambda\leq \int_{\Lambda} |\cS_\lambda(\alpha)|^2d\mu_\lambda\lesssim \|\alpha\|_{L^2(d\mu_\lambda)}^2=\|w\|_{\Hs}^2,
\end{equation}
proving the proposition.
\end{proof}
\begin{proof}\textit{of Proposition~\ref{fredprop}.} The result follows from Proposition~\ref{abcond} by using a simple Cauchy-Schwartz argument. Indeed, for any $\alpha\in L^2(d\mu_\lambda)$ and $v=(v_\lambda)_{\lambda\in\Lambda}\in\cE_\alpha$ we have that $\|v\|_{\bH}\leq\|\alpha\|_{L^2(d\mu_\lambda)}$, thus
\begin{align*}
|D_\lambda(v,v)|&\leq \Big(\int_{\Lambda^2} |\CalD(\lambda, \nu,\sigma)|^2\,d\mu_\nu\,d\mu_\sigma\Big)^{\frac{1}{2}}\Big(\int_{\Lambda^2} |v_\nu|^2|v_\sigma|^2 d\mu_\nu\,d\mu_\sigma\Big)^{\frac{1}{2}}\nonumber\\
&\leq \Big(\int_{\Lambda^2} |\CalD(\lambda, \nu,\sigma)|^2\,d\mu_\nu\,d\mu_\sigma\Big)^{\frac{1}{2}}\|v\|_{\bH}^2\leq \Big(\int_{\Lambda^2} |\CalD(\lambda, \nu,\sigma)|^2\,d\mu_\nu\,d\mu_\sigma\Big)^{\frac{1}{2}}\|\alpha\|_{L^2(d\mu_\lambda)}^2.
\end{align*}
It follows that
\begin{equation*}
\cS_\lambda(\alpha)\leq \Big(\int_{\Lambda^2} |\CalD(\lambda, \nu,\sigma)|^2\,d\mu_\nu\,d\mu_\sigma\Big)^{\frac{1}{2}}\|\alpha\|_{L^2(d\mu_\lambda)}^2\;\mbox{for any}\;\lambda\in\Lambda,\alpha\in L^2(d\mu_\lambda).
\end{equation*}
Integrating this inequality with respect to $\lambda\in\Lambda$ we obtain that
\begin{equation*}
\int_{\Lambda}|\cS_\lambda(\alpha)|^2d\mu_\lambda\leq \int_{\Lambda^3} |\CalD(\lambda, \nu,\sigma)|^2\,d\mu_\lambda\,d\mu_\nu\,d\mu_\sigma\|\alpha\|_{L^2(d\mu_\lambda)}^4\;\mbox{for any}\;\alpha\in L^2(d\mu_\lambda),
\end{equation*}
proving the proposition.
\end{proof}
\begin{remark}\label{r2.3}
We note that our Hilbert-Schmidt condition from \eqref{fredeq} does depend on the spectral decomposition of the self-adjoint operator $\Gamma$. Moreover, for each $\lambda\in\Lambda$ there exists $T_\lambda\in\cB(\bH)$ such that $D_\lambda(u,v)=\langle u,T_\lambda v\rangle$ for any $u,v\in\bH$. From \eqref{kerD} we immediately infer that
$\|T_\lambda\|_{\bH\to\bH}=\big(\int_{\Lambda^2}|\cD(\lambda,\nu,\sigma)|^2\,d\mu_\nu\,d\mu_\sigma\big)^{\frac{1}{2}}$ for any $\lambda\in\Lambda$. Therefore, \eqref{fredeq} is equivalent to
\begin{equation}\label{2.3-1}
\int_{\Lambda}\|T_\lambda\|_{\bH\to\bH}^2d\mu_\lambda<+\infty.
\end{equation}
Replacing the $\|\cdot\|_{\bH\to\bH}$ norm in \eqref{2.3-1} by the Hilbert-Schmidt norm we obtain an even stronger Hilbert-Schmidt condition on the bilinear map $D$
\begin{equation}\label{2.3-2}
\int_{\Lambda}\|T_\lambda\|_{\mathrm{HS}}^2d\mu_\lambda<+\infty,
\end{equation}
that can be shown to be independent of the spectral decomposition of $\Gamma$ or the choice of Hilbert bases on $\bH$.
\end{remark}
\begin{proof}\textit{of Proposition~\ref{abprop}.} Fix $v=(v_\lambda)_{\lambda\in\Lambda}\in\Hs$ and let $w\in\bH$ having spectral decomposition $(w_\lambda)_{\lambda\in\Lambda}$ where
$w_\lambda=\sup_{x\geq 0}|v_\lambda(x)|$.  One can readily check that $\|v\|_{\Hs}=\|w\|_{\bH}$. Using that $|D|$ defines a bounded bilinear map on $\bH$ we obtain that
\begin{align}\label{2.4-1}
\int_{\Lambda}\big(\sup_{x\geq 0}|D_\lambda(v(x),v(x))|\big)^2d\mu_\lambda&=\int_{\Lambda}\sup_{x\geq 0}\big|\int_{\Lambda^2} \CalD(\lambda,\nu,\sigma) v_\nu(x) v_\sigma(x)\,d\mu_\nu\,d\mu_\sigma           \big|^2d\mu_\lambda\nonumber\\
&\leq \int_{\Lambda} \Big(\int_{\Lambda^2}\sup_{x\geq 0} \big(|\CalD(\lambda,\nu,\sigma)|\,|v_\nu(x)|\, |v_\sigma(x)|\big)\,d\mu_\nu\,d\mu_\sigma\Big)^2d\mu_\lambda\nonumber\\
&=\int_{\Lambda} \Big(\int_{\Lambda^2}|\CalD(\lambda,\nu,\sigma)|\,w_\nu w_\sigma\,d\mu_\nu\,d\mu_\sigma\Big)^2d\mu_\lambda\nonumber\\
&=\int_{\Lambda}\big||D|_\lambda(w,w)\big|^2d\mu_\lambda=\big\||D|(w,w)\big\|_{\bH}^2\lesssim\|w\|_{\bH}^4=\|v\|_{\Hs}^4\nonumber\\
\end{align}
From \eqref{2.4-1} we have that $D_*(v,v)\in\Hs$ and
\begin{equation}\label{2.4-2}
\|D_*(v,v)\|_{\Hs}=\Big(\int_{\Lambda}\big(\sup_{x\geq 0}|D_\lambda(v(x),v(x))|\big)^2d\mu_\lambda\Big)^{\frac{1}{2}}\lesssim\|v\|_{\Hs}^2,
\end{equation}
proving the proposition.
\end{proof}
\subsection{Existence and uniqueness of an $\Hs$ stable manifold of equation \eqref{Gamma}}\label{stable-manifold}
Now we have all the ingredients needed to construct the stable manifold tangent at $u=0$ to the stable subspace of the linearized equation $u'=-\Gamma^{-1}u$. Throughout this subsection we assume that the space $\Hs$ is closed under the action $D_*$ in the sense of \eqref{closed-D*}.

Since $\Gamma$ is similar to the operator of multiplication by $(\gamma_\lambda)_{\lambda\in\Lambda}$ on $L^2(d\mu_\lambda)$, we can immediately infer that the stable/unstable subspace of equation $u'=-\Gamma^{-1}u$ is given by
$\bH_{\rms/\rmu}=\{h\in\bH:\supp(h_\lambda)\subseteq\Lambda_\pm\}$, where $\Lambda_\pm=\{\lambda\in\Lambda:\pm\gamma_\lambda>0\}$. Using that the linear operator $\Gamma$ is self-adjoint, one can readily check that $2\pi\rmi\omega\Gamma+I$ is invertible on $\bH$ for any $\omega\in\RR$ and $\sup_{\omega\in\RR}\|(2\pi\rmi\omega\Gamma+I)^{-1}\|<\infty$. Thus, the Fourier multiplier $K=(\Gamma\pa_x+I)^{-1}=\cF^{-1}M_{R}\cF$ is bounded on $L^2(\RR,\bH)$, where $M_R$ is the operator of multiplication on $L^2(\RR,\bH)$ by the operator valued function $R:\RR\to\cB(\bH)$ defined by $R(\omega)=2\pi\rmi\omega\Gamma+I$. Taking Fourier Transform in \eqref{Gamma} and then solving for $\cF u$ we can see that its $L^2$-solutions on $\RR_+$ satisfy \eqref{intvar}.

To solve \eqref{intvar} locally, we use a fixed point argument on a small closed ball centered at the origin in the weighted space $\bH_{*,\beta}=\{u:e^{\beta|\cdot|}u\in\Hs\}$ for $\beta\in [0,\|\Gamma\|^{-1})$. The procedure requires the following steps. First, using the representation of the stable semigroup $(T_\rms(x)h)_\lambda=e^{-\frac{x}{\gamma_\lambda}}h_\lambda$ for $x\geq 0$, $\lambda\in\Lambda$ and $h\in\bH_\rms$, we prove that $T_\rms(\cdot)h\in\bH_{*,\beta}$ and
\begin{equation}\label{trajectory-H}
\|T_\rms(\cdot)h\|_{\bH_{*,\beta}}\lesssim\|h\|_{\bH_\rms}\;\mbox{for any}\;h\in\bH_\rms.
\end{equation}
Next, we study the properties of the function $\Phi:\bH_\rms\times\bH_{*,\beta}\to\bH_{*,\beta}$ defined by $\Phi(h,u)=T_\rms(\cdot)h+\chi_+K(\chi_+D_*(u,u))$. Here we recall that $\chi_+$ is the characteristic function of $\RR_+$. To establish our results we need to prove that, provided $\Hs$ is closed under the action of $D_*$,
there exist $\eps_1>0$ and $\eps_2>0$ such that for any $\beta\in [0,\|\Gamma\|^{-1})$ the function
$\Phi$ maps $\ovB_{\bH_\rms}(0,\eps_1)\times\ovB_{\bH_{*,\beta}}(0,\eps_2)$ to $\ovB_{\bH_{*,\beta}}(0,\eps_2)$\footnote{$\ovB_\bX(0,\eps)$ denotes the closed ball in $\bX$ of radius $\eps$ centered at the origin.} and
\begin{equation}\label{Fix-point-estimate}
\|\Phi(h,u)-\Phi(h,v)\|_{\bH_{*,\beta}}\leq\frac{1}{2}\|u-v\|_{\bH_{*,\beta}}\;\mbox{for any}\;h\in\ovB_{\bH_\rms}(0,\eps_1),\,u,v\in\ovB_{\bH_{*,\beta}}(0,\eps_2).
\end{equation}
We note that for any weight $\beta\in [0,\|\Gamma\|^{-1})$ the space $\bH_{*,\beta}$ is closed under the action $D_*$ provided $\Hs$ is closed under the action of $D_*$. Therefore, to prove \eqref{Fix-point-estimate} it is enough to show that the Fourier multiplier $K$ can be extended to a bounded linear operator on $\bH_{*,\beta}$, for $\beta\in [0,\|\Gamma\|^{-1})$. This result follows by a long but fairly simple computation using the following convolution representation of $K$.
\begin{equation}\label{convolution-K}
(Kf)_\lambda(x)=\int_{\RR}\cK(x-y,\lambda)f_\lambda(y)\rmd y\;\mbox{for any}\;x\geq 0,\,\lambda\in\Lambda,\,f\in L^2(\RR_+,\bH)\cap L^\infty(\RR_+,\bH),
\end{equation}
where the kernel $\cK:\RR\times\Lambda\to\Lambda$ is given by
\begin{equation}\label{def-K-kernel}
\cK(x,\lambda)=\left\{\begin{array}{l l}
-\frac{1}{\gamma_\lambda}e^{-\frac{x}{\gamma_\lambda}} & \; \mbox{if $x>0$ and $\lambda\in\Lambda_+$ }\\
\frac{1}{\gamma_\lambda}e^{-\frac{x}{\gamma_\lambda}}  & \; \mbox{if $x<0$ and $\lambda\in\Lambda_-$ }\\
0& \; \mbox{otherwise}
\end{array} \right..
\end{equation}
Using a fixed point argument, from \eqref{Fix-point-estimate} we obtain that for any $h\in\ovB_{\bH_\rms}(0,\eps_1)$ equation $u=\Phi(h,u)$ has a \textit{unique}, local solution denoted $\bu(\cdot,h)$. Moreover, $\bu(\cdot,h)\in\ovB_{\bH_{*,\beta}}(0,\eps_2)$ depends smoothly on $h\in\ovB_{\bH_\rms}(0,\eps_1)$ in the $\bH_{*,\beta}$ norm. Using the representation \eqref{convolution-K} we have that
\begin{equation}\label{proj-sol}
\Pi_\rms\bu(0,h)=h\;\mbox{for any}\; h\in\ovB_{\bH_\rms}(0,\eps_1).
\end{equation}
Next, we introduce the stable manifold of equation \eqref{Gamma} by
\begin{equation}\label{def-Ms}
\cM_\rms=\{\bu(0,h):h\in\ovB_{\bH_\rms}(0,\eps_1)\}.
\end{equation}
From \eqref{proj-sol} we infer that $\cM_\rms=\mathrm{Graph}(\cJ_\rms)$, where $\cJ_\rms:\ovB_{\bH_\rms}(0,\eps_1)\to\bH_\rmu$ by $\cJ_\rms(h)=\Pi_\rmu\bu(0,h)$. Moreover, since $(Kf)(\cdot+x_0)=Kf(\cdot+x_0)$ for any $x_0>0$ and $f\in\bH_{*,\beta}$, by using the uniqueness of solution of equation $u=\Phi(h,u)$, we conclude that the manifold $\cM_\rms$ is invariant under the flow of equation \eqref{Gamma}. Finally, by differentiating with respect to $h$ in \eqref{intvar}, we infer that $\cJ_\rms'(0)=0$ proving that the manifold $\cM_\rms$ is tangent at $u=0$ to the stable subspace $\bH_\rms$.

\section{Counterexamples}\label{counterexamples}

In the remainder of the paper, we provide counterexamples showing that (i) boundedness of $D$ does not imply
boundedness of $|D|$ (Subsection~\ref{ceg1}) and (ii) boundedness of $D$ from $\bH\times\bH \to \bH$ does not imply $\Hs$ is closed under the action of $D_*$. (Subsection~\ref{ceg2}).
For simplicity, we work in the discrete case $L^2(d\mu_\lambda)=\ell^2$; however, the examples have obvious continuous
counterparts.

\subsection{Counterexample (i)}\label{ceg1}
Consider the matrix
\be\label{mat}
M_1(\theta):=\bp \cos \theta & \sin \theta\\ \sin \theta & -\cos \theta \ep,
\qquad 0<\theta< \frac{\pi}{2}.
\ee
This has eigenvalues $\pm 1$, yet, separating positive and negative parts
\be\label{pn}
M_1^+(\theta):=\bp \cos \theta & \sin \theta\\ \sin \theta & 0 \ep,
\qquad
M_1^+(\theta):=\bp
0&0\\0 & \cos \theta \ep,
\ee
we see that there exists $\theta_0\in (0,\frac{\pi}{4})$ such that $M_1^+(\theta_0)$ has an eigenvalue $\rho>1$, with an associated eigenvector $q=(q_1,q_2)^{\mathrm{T}}$ (by nonnegative version of Frobenius--Perron)
having nonnegative eigenvalues (indeed, calculation shows that the entries are strictly positive).

Next, we recall the following decomposition: $\ell^2=\bigoplus_{j=1}^{\infty}\RR^{2^j}$. We construct a linear operator $M(\theta)=\bigoplus_{j=1}^{\infty}M_j(\theta)\in\cB(\bigoplus_{j=1}^{\infty}\RR^{2^j})$, that is an infinite matrix, recursively, defining its upper lefthand $2^j\times 2^j$ block $M_j(\theta)$, $j\geq1$, as follows
\be\label{recursion}
M_{j+1}(\theta):=\bp \cos\theta\,M_j(\theta) & \sin\theta\,M_j(\theta)\\ \sin \theta\,M_j(\theta)& -\cos\theta\,M_j(\theta)\ep,
\ee
so that
\be\label{recursionpm}
M_{j+1}^+(\theta):=\bp \cos\theta\,M_j^+(\theta) & \sin\theta\,M_j^+(\theta)\\ \sin\theta\,M_j^+(\theta)& -\cos\theta\,M_j^-(\theta)\ep,
\;
M_{j+1}^-(\theta):=\bp \cos\theta\, M_j^-(\theta) & \sin\theta\,M_j^-(\theta)\\ \sin\theta\,M_j^-(\theta)& -\cos\theta\,M_j^+(\theta)\ep.
\ee
\begin{lemma}\label{bound-M} For any $j\geq1$ and $\theta\in (0,\frac{\pi}{2})$ the matrix $M_j(\theta)$ is an isometry on $\RR^{2^j}$, therefore the linear operator $M(\theta)=\bigoplus_{j=1}^{\infty}M_j(\theta)$ is bounded on $\ell^2$.
\end{lemma}
\begin{proof}
A direct computation shows that $\|M_1(\theta)v\|_{\RR^2}=\|v\|_{\RR^2}$ for any $v\in\RR^2$, showing that $M_1(\theta)$ is an isometry on $\RR^2$. Assume that $M_k(\theta)$ is an isometry on $\RR^{2^k}$. Let $w=(w_1,w_2)\in\RR^{2^{k+1}}=\RR^{2^k}\times\RR^{2^k}$. From \eqref{recursion} we obtain that
\begin{align}\label{bm-1}
\|M_{k+1}&(\theta)w\|_{\RR^{2^{k+1}}}^2=\|\cos\theta\,M_k(\theta)w_1+\sin\theta\,M_k(\theta)w_2\|_{\RR^{2^k}}^2+\|\sin\theta\,M_k(\theta)w_1-\cos\theta\,M_k(\theta)w_2\|_{\RR^{2^k}}^2\nonumber\\
&=\cos^2\theta\|M_k(\theta)w_1\|_{\RR^{2^k}}^2+\sin^2\theta\|M_k(\theta)w_2\|_{\RR^{2^k}}^2+2\sin\theta\,\cos\theta\langle M_k(\theta)w_1,M_k(\theta)w_2\rangle_{\RR^{2^k}}\nonumber\\
&\quad+\sin^2\theta\|M_k(\theta)w_1\|_{\RR^{2^k}}^2+\cos^2\theta\|M_k(\theta)w_2\|_{\RR^{2^k}}^2-2\sin\theta\,\cos\theta\langle M_k(\theta)w_1,M_k(\theta)w_2\rangle_{\RR^{2^k}}\nonumber\\
&=\|M_k(\theta)w_1\|_{\RR^{2^k}}^2+\|M_k(\theta)w_2\|_{\RR^{2^k}}^2=\|w_1\|_{\RR^{2^k}}^2+\|w_2\|_{\RR^{2^k}}^2=\|w\|_{\RR^{2^{k+1}}}^2,
\end{align}
proving that $M_{k+1}(\theta)$ is an isometry on $\RR^{2^{k+1}}$.
\end{proof}
From Lemma~\ref{bound-M} we can immediately infer that the map $D:\ell^2\times\ell^2\to\ell^2$ defined by
$D(u,v):=\langle u, M(\theta_0)v\rangle_{\ell^2}\mathrm{e}_1$ is a bounded bilinear map. Here $\{\mathrm{e}_n\}_{n\geq}$ denotes the standard orthonormal Hilbert basis of $\ell^2$.
However, $|D|$ may, by recursion, be easily seen to be unbounded, by application to the
nonnegative-entry (in fact, strictly positive entry) test vectors
$u_j\in\RR^{2^j\times 2^j}$, $j\geq1$, determined recursively by
\be\label{recursionW}
u_{j+1}=\bp q_1 u_j\\q_2 u_j\ep,\;\mbox{for}\;j\geq1
\ee
with $u_1=q\in\RR^2$. We note that the following identity holds true:
\begin{equation}\label{|D|-identity}
|D|(u,v)=\langle u,|M_j(\theta_0)|v\rangle_{\RR^{2^j}}\;\mbox{for any}\;u,v\in\RR^{2^j\times 2^j}\subset\ell^2,j\geq 1.
\footnote{For a matrix $A$, we denote by $|A|:=A^++A^-$, where $A^\pm$ are the positive/negative parts of $A$.}
\end{equation}
Using\eqref{|D|-identity}, we prove an estimate that will allows to immediately conclude that $|D|$ is not bounded on $\bH=\ell^2$.
\bl\label{bd1}
For $M\in\cB(\ell^2)$ defined as in \eqref{recursion}, the following inequalities hold true:
\be\label{bd}
\langle u_j,|M_j(\theta_0)| u_j\rangle_{\RR^{2^j}} \geq \rho^j \|u_j\|_{\RR^{2^j}}^2.
\ee
\el
\begin{proof}
Estimate \eqref{bd} holds with equality for $j=1$, by construction. Assume that it holds for $j\leq k$. Computing
$$
\begin{aligned}
\langle u_{k+1},&|M_{k+1}(\theta_0)| u_{k+1}\rangle_{\RR^{2^{k+1}}}=
\bp q_1 u_k\\q_2 u_k\ep^\mathrm{T}
\bp \cos\theta\,M_k^+(\theta_0) & \sin\theta\,M_k^+(\theta_0)\\ \sin\theta\,M_k^+(\theta_0)& \cos\theta\,M_k^-(\theta_0)\ep
\bp q_1 u_k\\q_2 u_k\ep \\
&=
\langle u_k, M_k^+(\theta_0),u_k\rangle_{\RR^{2^k}} \langle q,M_1^+(\theta_0) q\rangle_{\RR^2} +
\langle u_k, M_k^-(\theta_0),u_k\rangle_{\RR^{2^k}} \langle q,M_1^-(\theta_0) q\rangle_{\RR^2},
\end{aligned}
$$
and using nonnegativity of entries to see that
$ \langle u_k, M_k^-(\theta_0),u_k\rangle_{\RR^{2^k}}\langle q,M_1^-(\theta_0)q\rangle_{\RR^2}\geq 0$, we obtain using the induction hypothesis that
\begin{align*}
\langle u_{k+1},|M_{k+1}(\theta_0)| u_{k+1}\rangle_{\RR^{2^{k+1}}}&
\geq
\langle u_k, M_k^+(\theta_0),u_k\rangle_{\RR^{2^k}} \langle q,M_1^+(\theta_0) q\rangle_{\RR^2}\\
&\geq \rho^k \|u_k\|_{\RR^{2^k}}^2 \rho \|q\|_{\RR^2}^2= \rho^{k+1}\|u_{k+1}\|_{\RR^{2^{k+1}}}^2,
\end{align*}
proving the claim.
\end{proof}

This shows it is not true that $|D|$ is bounded when $D$ is bounded, even restricted to a single coordinate.  However,
this property does not imply that $\Hs=\{f=(f_n)_{n\geq 1}:\sup_{x\geq 0}|f_n(x)|\in\ell^2\}$ is not closed in the sense of \eqref{closed-D*} under the action of
the associated extension $D_*$. It is a necessary but not sufficient condition for a counterexample to that more primary question.
And, indeed, it is easily seen that this cannot be violated by a bilinear map involving a single mode.

\subsection{Counterexample (ii)}\label{ceg2}
Finally, we give a (vectorial) counterexample showing that $D$ bounded from $\bH\times \bH\to \bH$
does not imply that $\Hs$ is closed under the action of $D_*$. Similar to the previous counterexample we take $\bH=\ell^2=\bigoplus_{j=1}^{\infty}\RR^{2^j}$. Next, we introduce the matrix $T_j:=M_j(\frac{\pi}{4})\in\RR^{2^j\times 2^j}$, $j\geq 1$, where $M_j(\theta)$ is defined recursively by \eqref{mat} and \eqref{recursion}. Let $w_j=(w_{j,k})_{1\leq k\leq 2^j}\in\RR^{2^j}$ be the vector defined by $w_{j,k}=2^{-j/2}$, $1\leq k\leq 2^j$, $j\geq 1$. A direct computation shows that
\begin{equation}\label{est-w}
\|w_j\|_{\RR^{2^j}}=1\;\mbox{for any}\;j\geq1.
\end{equation}
Moreover, from Lemma~\ref{bound-M} we have that
\begin{equation}\label{bound-Tj}
\|T_jv\|_{\RR^{2^j}}=\|v\|_{\RR^{2^j}}\;\mbox{for any}\; v\in\RR^{2^j},\,j\geq1.
\end{equation}
We define $D_j:\RR^{2^j}\times\RR^{2^j}\to\RR^{2^j}$ be the bilinear form defined by $D_j(u,v)=\langle u,w_j\rangle T_jv$. From \eqref{est-w} and \eqref{bound-Tj} we obtain that
\begin{equation}\label{Dj-est}
\|D_j(u,v)\|_{\RR^{2^j}}=|\langle u,w_j\rangle|\,\|T_jv\|_{\RR^{2^j}}\leq \|u\|_{\RR^{2^j}}\|w_j\|_{\RR^{2^j}}\|v\|_{\RR^{2^j}}=\|u\|_{\RR^{2^j}}\|v\|_{\RR^{2^j}}
\end{equation}
for any $u,v\in\RR^{2^j}$ and $j\geq 1$. It follows that the bilinear map $D:\bigoplus_{j=1}^{\infty}\RR^{2^j}\times\bigoplus_{j=1}^{\infty}\RR^{2^j}\to\bigoplus_{j=1}^{\infty}\RR^{2^j}$ defined by
\begin{equation}\label{def-D-c2}
D\big( u, v\big)=\big(D_j(u_j,v_j)\big)_{j\geq 1}\;\mbox{for}\;u=(u_j)_{j\geq 1}, v=(v_j)_{j\geq 1}\in\bigoplus_{j=1}^{\infty}\RR^{2^j}
\end{equation}
is well-defined and bounded. Indeed, from \eqref{Dj-est} and the Cauchy-Schwartz inequality we can immediately infer that $\|D(u,v)\|_{\ell^2}\leq\|u\|_{\ell^2}\|v\|_{\ell^2}$, for any $u,v\in\ell^2=\bigoplus_{j=1}^{\infty}\RR^{2^j}$.

Below we use Proposition~\ref{abcond} to prove $\Hs=\{f=(f_n)_{n\geq 1}:\sup_{x\geq 0}|f_n(x)|\in\ell^2\}$ is not closed in the sense of \eqref{closed-D*} under the action of the associated extension $D_*$. We introduce the vectors $\alpha_j=(\alpha_{j,k})_{1\leq k\leq 2^j}\in\RR^{2^j}$, $j\geq1$ by $\alpha_{j,k}=\frac{1}{j2^{j/2}}$. Thus, $\|\alpha_j\|_{\RR^{2^j}}=\frac{1}{j}$ for any $j\geq 1$, which implies that $\alpha=(\alpha_j)_{j\geq 1}\in\ell^2=\bigoplus_{j=1}^{\infty}\RR^{2^j}$ and $\|\alpha\|_{\ell^2}^2=\sum_{j=1}^{\infty}\|\alpha_j\|_{\RR^{2^j}}^2=\sum_{j=1}^{\infty}\frac{1}{j^2}<\infty$. In the case at hand we have that
\begin{equation}\label{E-alpha}
\cE_\alpha=\Big\{v=(v_j)_{j\geq 1}\in\bigoplus_{j=1}^{\infty}\RR^{2^j}:v_j=(v_{j,k})_{1\leq k\leq 2^j}\in\RR^{2^j},\;|v_{j,k}|\leq\frac{1}{j2^{j/2}},\;1\leq k\leq 2^j, j\geq 1\Big\}.
\end{equation}
Next, we denote by $D_{j,k}:\RR^{2^j}\times\RR^{2^j}\to\RR$, $1\leq k\leq 2^j$, $j\geq 1$ the components of the bilinear map $D_j$, $j\geq 1$. A crucial observation is that all the entries of $T_j$ on the first row are equal to $2^{-j/2}$ for any $j\geq 1$. It follows that
\begin{equation}\label{Dj-1}
D_{j,1}(v,v)=2^{-j}\Big(\sum_{k=1}^{2^j}v_{j,k}\Big)^2\;\mbox{for any}\;v\in\cE_\alpha.
\end{equation}
Similarly, for any $k\in\{2,\dots,2^j\}$ the $k$-th row of $T_j$ consists of $2^{j-1}$ entries equal to $2^{-j/2}$ and $2^{j-1}$ entries equal to $-2^{-j/2}$. Thus, for each $k\in\{2,\dots,2^j\}$ there exits $\sigma_k$ a permutation of the set $\{1,\dots,2^j\}$ such that
\begin{equation}\label{Dj-k}
D_{j,k}(v,v)=2^{-j}\Big(\sum_{m=1}^{2^{j-1}}v_{j,\sigma_k(m)}\Big)^2-2^{-j}\Big(\sum_{m=2^{j-1}+1}^{2^j}v_{j,\sigma_k(m)}\Big)^2\;\mbox{for any}\;v\in\cE_\alpha.
\end{equation}
From \eqref{Dj-1} and \eqref{Dj-k} we obtain that
\begin{equation}\label{sup-Djk}
\sup_{v\in\cE_\alpha}|D_{j,1}(v,v)|=\frac{1}{j^2},\;\sup_{v\in\cE_\alpha}|D_{j,k}(v,v)|=\frac{1}{4j^2}\;\mbox{for any}\;2\leq k\leq 2^j,\,j\geq1.
\end{equation}
Denoting by $\cS_j(\alpha)=(\sup_{v\in\cE_\alpha}|D_{j,k}(v,v)|)_{1\leq k\leq 2^j}\in\RR^{2^j}$, we have that
\begin{equation}\label{Sj-norm}
\|\cS_j(\alpha)\|_{\RR^{2^j}}^2=\frac{1}{j^4}+\frac{2^j-1}{16j^4}\geq\frac{2^j}{16j^4}\;\mbox{for any}\; j\geq 1,
\end{equation}
which implies that $\cS(\alpha)=(\cS_j(\alpha))_{j\geq 1}\not\in\bigoplus_{j=1}^{\infty}\RR^{2^j}=\ell^2$. From Proposition~\ref{abcond} we infer that $\Hs$ is not closed under the action of $D_*$.

\section{Discussion and open problems}\label{discussion}
There are two important differences between our analysis here in Section
\ref{stable-manifold} and the $H^1$-stable manifold construction of \cite{PZ1}.
First, at linear level, trajectories $T_s(\cdot)h\in\bH_{*,\beta}$ for any $h\in \bH_\rms$.
Therefore the condition $u(0)\in \dom(\Gamma^{-1/2})$ is not needed for membership in the stable subspace.
Second, at nonlinear level, we may express the stable manifold simply as a graph over the stable manifold,
requiring only $\Pi_\rms u\in \Hs$, whereas in \cite{PZ1} we required the more complicated implicit condition
$$
\Pi_\rms(u + D(u,u))\in \dom(\Gamma^{-1/2}) \Leftrightarrow
\Pi_\rms u' \in \dom(\Gamma^{-1/2}).
$$
This greatly simplifies the argument at the same time that it extends the results.

On the other hand, the analysis of \cite{PZ1} applied to the important case of the steady Boltzmann equation
with hard-sphere collision kernel, the main example and motivation for our investigations.
To the contrary, our Hilbert--Schmidt condition derived here does not hold for Boltzmann's equation, and
it is not at all clear how one would check that absolute boundedness condition on the kernel.
It might be that one could show boundedness of $D_*$ directly for Boltzmann's equation, however, using
the explicit structure of the collision operator (for example, the linearized collision operator may
be expressed \cite{G,Z2} as the sum of a positive real-valued multiplication operator bounded above and
below, and a compact operator $\check K$ that is readily seen to satisfy the Hilbert-Schmidt condition).
This would be a very interesting open problem to resolve.

Whether or not one can verify the bounded-$D_*$ condition, answerering the question of existence of
a ``full'' stable manifold, there remains the second question whether solutions small in $L^\infty(
\R_+,\bH)$ necessarily decay exponentially.
As noted in \cite{PZ1},
a very interesting observation due to Fedja Nazarov \cite{Na} based on
the indefinite Lyapunov functional relation $\langle u,\Gamma u\rangle'=-\|u\|_{\bH}^2+o(\|u\|_{\bH}^2)$ yields
the $L^2$-exponential decay result $e^{\beta |\cdot|}\|u(\cdot)\|\in L^2(\R_+)$ for some $\beta>0$, hence (by interpolation) in any
$L^p(\RR_+)$, $2\leq p<\infty$. However, it is not clear what happens in the critical norm $p=\infty$.
This, too, would be very interesting to resolve, either exhibiting a counterexample or proving decay.

Another approach to construction of a ``full'' stable manifold for Boltzmann's equation, as discussed,
e.g., in \cite{LiuYu,Z}, is to work in an appropriately weighted $L^\infty(\R_+,\bB)$, where $\bB$ is some weighted $L^\infty$ space in
$\xi$ (here $\xi$ denotes the independent variable of velocity, as standard for
Boltzmann's equation \cite{LiuYu,PZ1}), for which boundedness of $D_*$ would follow immediately from
boundedness of $D$, as would exponential decay of solutions merely small in $L^\infty(\R_+)$,
answering both questions posed here.
However, to date, it has not been shown that $D$ is bounded in this setting,
and it is not clear whether or not this is true; see \cite{Z}.
This is another open problem that would be very interesting to resolve.

\end{document}